\documentclass[a4paper,12pt,draft]{amsart}
\usepackage[margin=30mm]{geometry}
\usepackage{amssymb, amsmath, amsthm, amscd}

\usepackage[T1]{fontenc}
\usepackage{eucal,mathrsfs,dsfont}
\usepackage{color}

\definecolor{mno}{rgb}{0.5,0.1,0.5}

\newcommand{\R}{\mathds R}

\newcommand{\Ss}{\mathds S}

\newcommand{\Pp}{\mathds P}
\newcommand{\Ee}{\mathds E}

\newcommand{\I}{\mathds 1}

\def\<{\langle}
\def\>{\rangle}

\newtheorem{theorem}{Theorem}[section]
\newtheorem{lemma}[theorem]{Lemma}
\newtheorem{proposition}[theorem]{Proposition}
\newtheorem{corollary}[theorem]{Corollary}

\theoremstyle{definition}

\newtheorem{remark}[theorem]{Remark}

\makeatletter 

\@addtoreset{equation}{section}
\makeatother 

\begin{document}
\allowdisplaybreaks
\title[Spatial regularity of L\'{e}vy type operators]
{\bfseries Spatial regularity of semigroups generated by L\'{e}vy type operators}

\author{Mingjie Liang \qquad Jian Wang}
\thanks{\emph{M.\ Liang:}
College of Mathematics and Informatics, Fujian Normal University,
350007 Fuzhou, P.R. China. \texttt{liangmingjie@aliyun.com}}
\thanks{\emph{J.\ Wang:}
College of Mathematics and Informatics \& Fujian Key Laboratory of Mathematical Analysis and Applications (FJKLMAA), Fujian Normal University, 350007 Fuzhou, P.R. China. \texttt{jianwang@fjnu.edu.cn}}
\date{}

\begin{abstract}We apply the probabilistic coupling approach to establish the spatial regularity of semigroups associated with L\'{e}vy type operators, by assuming that the martingale problem of L\'{e}vy type operators is well posed. In particular, we can prove the Lipschitz continuity of the semigroups under H\"{o}lder continuity of coefficients, even when the L\'evy kernel corresponding to L\'{e}vy type operators is singular.

\medskip


\medskip

\end{abstract}

\maketitle
\allowdisplaybreaks

\section{Introduction and Main Results}\label{section1}
We consider the following L\'evy type operator
\begin{equation}\label{levytype}Lf(x)=\int\!\!\Big(f(x+z)-f(x)-\langle\nabla f(x), z\rangle\I_{B(0,1)}(z)\Big)c(x,z)\,\nu(dz),\end{equation}
where $\nu$ is a L\'evy measure, i.e., $\nu(\{0\})=0$ and $\int
(1\wedge |z|^2)\,\nu(dz)<\infty$, and $(x,z)\mapsto c(x,z)$ is a
continuous function such that $c(x,z)\in(c_*,c^*)$ for some
constants $0<c_*\le c^*<\infty$. The operator $L$ is a non-local
version of the classic second order elliptic operator with
non-divergence form, and it has been attracted a lot of attentions
in the community of analysis and PDEs, see, e.g.\ \cite{CS, CS1}. In
the probability theory, $L$ is corresponding to an infinitesimal
generator of a large class of Feller processes, see the monographs
\cite{BSW, Ja}.  The operator $L$ is also connected with the
following pure jump process on $\R^d$:
\begin{equation}\label{jumps}\begin{split}dJ_t=&\int_0^\infty \int_{\{|z|\le 1\}}
z\I_{[0,c(J_{t-},z)]}(r)\,\widetilde N(dz,dr,dt)\\
&+\int_0^\infty \int_{\{|z|> 1\}} z\I_{[0,c(J_{t-},z)]}(r)\,
N(dz,dr,dt),\quad t\ge0,\end{split}\end{equation} where
$N(dz,dr,dt)$ is a Poisson random measure on
$\R^d\times[0,\infty)\times[0,\infty)$ with intensity measure
$\nu(dz)\,dr\,dt$, and $\widetilde
N(dz,dr,dt)=N(dz,dr,dt)-\nu(dz)\,dr\,dt$ is the compensated Poisson
random measure. See \cite{KU, KP, X, XX} for more details. The pure
jump process of the form as $(J_t)_{t\ge0}$ has been found very
useful in applications; for instance, $(J_t)_{t\ge0}$ plays a
crucial role as the control when proving Freidlin-Wentzell type
large deviation for stochastic differential equations with jumps via
weak convergence approach, see e.g. \cite{BCD}.

The aim of this paper is to establish the spatial regularity of semigroups associated with the operator $L$ given by \eqref{levytype}.
We will adopt the probabilistic coupling approach, which recently has been extensively studied in \cite{MW, LW14, Lwcw, Wang15}. To study analytic properties for L\'evy type operators via probabilistic method, as one of the standing assumptions,
the existence of a strong Markov process associated with L\'evy type operators was assumed. For example, see \cite{BK1, BK2, BL, LW14,Wang15}, where the strong Markov property played an important role. Similarly, throughout this paper we shall assume that the martingale problem for $(L,C_b^2(\R^d))$ is well posed (see Section \ref{section2} below for its definition). In particular, there is a strong Markov process $X:=((X_t)_{t\ge0}, (\Pp^x)_{x\in\R^d})$, whose generator is just the L\'evy type operator $L$. Below, for any $f\in B_b(\R^d)$ (the set of
bounded measurable functions on $\R^d$), let
$$P_tf(x)=\Ee^xf(X_t),\quad x\in \R^d,t\ge0$$ be the semigroups corresponding to the operator $L$.

 \ \

To state the contribution of our paper, we consider the following L\'evy measure partly motivated by \cite{KRS}. (Actually, the paper \cite{KRS} treated L\'evy kernel case but with a fixed order, see \cite[(1.2)]{KRS}.)
Suppose that there are constants $c_1,c_2>0$ and $0<\alpha_1\le \alpha_2<2$ such that
\begin{equation}\label{e:mee}\frac{c_1}{|z|^{d+\alpha_1}}\I_{V_\xi}(z)\,dz\le \nu(dz)\le \frac{c_2}{|z|^{d+\alpha_2}}\,dz,\end{equation}  where \begin{equation}\label{e:mea}V_\xi=\{z\in \R^d: |z|\le 1\,\,\text{and}\,\, \langle z,\xi\rangle \ge \delta |z|\}\end{equation} with $\xi \in \Ss^{d-1}$ and constant $\delta\in (0,1)$.

\begin{theorem}\label{main} Under the assumption \eqref{e:mee}, we define for any $r>0$,
$$
  w(r)=
    \begin{cases}
    \sup\limits_{x,y\in \R^d: |x-y|=r}\displaystyle\int_{\{|z|\le 1\}}|z|^2 |c(x,z)-c(y,z)|\,\nu(dz), & \alpha_2\in [1,2);\\
     \sup\limits_{x,y\in \R^d: |x-y|=r}\displaystyle\int_{\{|z|\le 1\}}|z| |c(x,z)-c(y,z)|\,\nu(dz), & \alpha_2\in (0,1).
    \end{cases}
 $$ Then the following statements hold.
\begin{itemize}
\item[(1)] If $\alpha_1\in (1,2)$ and
\begin{equation}\label{e:ffeesss}\lim_{r\to0} w(r) r^{\alpha_1-2}\log^{1+\theta}(1/r)=0\end{equation} for some $\theta>0$,
then there exists a constant $C>0$ such that for all $f\in B_b(\R^d)$ and $t>0$,
$$
  \sup_{x\neq y}\frac{{|P_t f(x)-P_t f(y)|}}{|x-y|}
  \le C\|f\|_\infty {(t\wedge1)^{-1/\alpha_1}} \bigg[\log^{(1+\theta)/\alpha_1}\left(\frac{1}{t\wedge e}\right)\bigg].
  $$
\item[(2)] If $\alpha_1\in [1,2)$ and
\begin{equation}\label{e:ffeesss11}\lim_{r\to0} w(r) r^{\alpha_1-2}\log(1/r)=0,\end{equation}
then, for any $\theta>\I_{\{\alpha_1=1\}}$, there exists a constant
$C>0$ such that for all $f\in B_b(\R^d)$ and $t>0$,
$$
  \sup_{x\neq y}\frac{{|P_t f(x)-P_t f(y)|}}{|x-y|\big|\log
  |x-y|\big|^{\theta}}
  \le C\|f\|_\infty {(t\wedge1)^{-1/\alpha_1}} \bigg[\log^{-\theta+(1/\alpha_1)}\left(\frac{1}{t\wedge e}\right)\bigg].
  $$

\item[(3)] If $\alpha_2\in (0,1)$ and
\begin{equation}\label{e:ffeesss1}\lim_{r\to0} w(r) r^{\alpha_1-1}=0,\end{equation}
then for any $\theta\in (0,\alpha_1)$, there exists a constant $C>0$
such that for all $f\in B_b(\R^d)$ and $t>0$,
$$
  \sup_{x\neq y}\frac{{|P_t f(x)-P_t f(y)|}}{|x-y|^\theta}
  \le C\|f\|_\infty{(t\wedge1)^{-\theta/\alpha_1}}.
  $$
\end{itemize}\end{theorem}

Note that, the continuity assumptions on $w(r)$ in the theorem above
are weaker than those on the function $x\mapsto c(x,z)$ (uniformly
with respect to $z$). The latter was used to study the H\"{o}lder
continuity of solutions to a class of second order non-linear
elliptic integro-differential equations in \cite[Theorem 3.1]{BCI}.
See also \cite[Section 4.2]{BCI2} for more details. We also mention
that similar assumptions (but a little stronger) on $w(r)$ have been
adopted to study the pathwise uniqueness of solution to stochastic
differential equations driven by the pure jump process
$(J_t)_{t\ge0}$ of the form \eqref{jumps}, see \cite[Theorem
3.1]{KU}, \cite[Theorem 1.2]{X} or \cite[Theorem 2.1]{XX}.

\subsection{Applications} A Borel measurable function $u$ on $\R^d$
is called harmonic with respect to $L$, if $P_tu(x)=u(x)$ for all
$x\in \R^d$ and $t>0$. The following result is a direct consequence
of Theorem \ref{main}.
\begin{corollary}\label{C:cor1} Under the setting of Theorem $\ref{main}$, we have the following three statements.
\begin{itemize}
\item[(1)] If assumptions in the first assertion of Theorem $\ref{main}$ hold, then
any bounded measurable function $u$ is Lipschitz continuous.

\item[(2)] If assumptions in the second assertion of Theorem $\ref{main}$ hold, then, for any
$\theta>\I_{\{\alpha_1=1\}}$, any bounded measurable function $u$ is $r\log^\theta |r|$-order
continuous.

\item[(3)] If assumptions in the third assertion of Theorem $\ref{main}$ hold, then for any $\varepsilon>0$,
any bounded measurable function $u$ is
$(\alpha_1-\varepsilon)$-H\"{o}lder continuous.
\end{itemize}\end{corollary}

Next, we apply Theorem \ref{main} to prove the following Liouville
theorem. See \cite[Theorem 2.1]{RS} for the related discussion for
general symmetric $\alpha$-stable operators.

\begin{corollary}\label{C:cor2} Consider the setting of Theorem $\ref{main}$, and let $u$ be a harmonic function respect to the operator $L$. Suppose that there exists a constant $c>0$ such that for
all $r\ge1$, $$\|u\|_{L^\infty(B_r(0),dx)}\le c r^\beta,$$ where
$B_r(0)=\{z\in \R^d:|z|<r\}$ and $0\le \beta<\alpha_2$. If
assumptions in any assertion of Theorem $\ref{main}$ hold, then $u$
is a polynomial of degree at most $\lfloor \beta\rfloor$, where
$\lfloor x\rfloor$ denotes the integer part of $x$.
\end{corollary}

\subsection{Perturbation result} To further illustrate the power of the
coupling approach, we next give a perturbation result corresponding
to Theorem \ref{main}. Below we will consider the following L\'evy
type operator
\begin{equation}\label{lvy1}
L_*f(x)=\int\!\!\Big(f(x+z)-f(x)-\langle\nabla f(x),
z\rangle\I_{B(0,1)}(z)\Big)\left(c(x,z)\,\nu(dz)+\mu(x,dz)\right),\end{equation}
where $c(x,z)$ and $\nu(dz)$ are assumed same as those in the
beginning of this section, and $\mu(x,dz)$ is a L\'evy kernel on
$\R^d$ satisfying that $$\sup_{x\in \R^d}\int
(1\wedge|z|^2)\,\mu(x,dz)<\infty$$ and for any $h\in C_b^2(\R^d)$,
the function $$x\mapsto \int_{\R^d} h(z)
\frac{|z|^2}{1+|z|^2}\,\mu(x,dz)$$ is continuous. We assume that the
martingale problem for $(L_*, C_b^2(\R^d))$ is well-posed. Clearly,
the operator $L_*$ is just the operator $L$ perturbed by the L\'evy
kernel $\mu(x,dz)$. We emphasize that, we do not assume that the
L\'evy kernel $\mu(x,dz)$ is absolutely continuous with respect to
the Lebesgue measure.

For any $r>0$, we define
  $$w_\mu(r)=\sup_{x,y\in \R^d:|x-y|=r}\int_{\{|z|\le 1\}}|z|^2|\mu(x,dz)-\mu(y,dz)|;$$ if additionally  \begin{equation}\label{eefs}\sup_{x\in \R^d}\int_{\{|z|\le 1\}}|z|\,\mu(x,dz)<\infty,\end{equation} then $w_\mu(r)$ above is replaced by
  $$w_\mu(r)=\sup_{x,y\in \R^d:|x-y|=r}\int_{\{|z|\le 1\}}|z||\mu(x,dz)-\mu(y,dz)|.$$

\begin{theorem}\label{main0} Under assumptions of Theorem $\ref{main}$ and notations above, define
$$w_*(r)= w(r)+w_\mu(r),\quad r>0.$$
Then
\begin{itemize}
\item[(1)] the first assertion of Theorem $\ref{main}$ holds, if $\alpha_1\in(1,2)$ and \eqref{e:ffeesss} holds with $w_*(r)$ replacing $w(r)$.

\item[(2)] the second assertion of Theorem $\ref{main}$ holds, if $\alpha_1\in[1,2)$ and \eqref{e:ffeesss11} holds with $w_*(r)$ replacing $w(r)$.

\item[(3)] the third assertion of Theorem $\ref{main}$ holds, if $\alpha_2\in (0,1)$, \eqref{eefs} is satisfied, and \eqref{e:ffeesss1} holds with $w_*(r)$ replacing $w(r)$.

\end{itemize}\end{theorem}

The remainder of this paper is arranged as follows. The next section
is devoted to the construction of a new coupling operator for the
L\'evy type operator $L$ given by \eqref{levytype}, and the
existence of coupling process on $\R^{2d}$ associated with the
constructed coupling operator. In Section \ref{section3}, we first
present some preliminary estimates for the coupling operator, and
then give general results for the regularity of associated
semigroups, by making full use of the coupling operator and the
coupling process constructed in Section \ref{section2000}. Finally,
the proofs of all results above are given in the last part of
Section \ref{section3}.

\section{Coupling operator and coupling process for L\'evy type operators}\label{section2000}
This section is split into two parts. We first present a new coupling operator for the L\'evy type operator $L$ given by \eqref{levytype}, and then prove the existence of coupling process on $\R^{2d}$ associated with the constructed coupling operator.
\subsection{Coupling operator for L\'evy type operator} The construction below is heavily based on the refined basic coupling for stochastic differential equations driven by additive L\'evy noises first introduced in \cite{Lwcw}. See \cite{MW} for the recent study on stochastic differential equations driven by multiplicative L\'evy noises. However, the L\'evy type operator $L$ given by \eqref{levytype} essentially is different from stochastic differential equations with jumps, and the main difficulty here is due to that the coefficient $c(x,z)$ in the operator $L$ depends on both space variables, which requires a new idea for the construction of a coupling operator.

For any $x,y,z, u\in \R^d$, define \begin{equation}\label{e:not}
\begin{split}
c(x,y,u,z)&=c(x,z)\wedge c(y,z)\wedge c(x,z-u)\wedge c(y,z-u),\\
\nu_u(dz)&=\nu\wedge (\delta_u*\nu)(dz),\\
\mu_{x,y,u}(dz)&= c(x,y,u,z)\,\nu_u(dz),\\
\tilde\nu_{x,y}(dz)&=(c(x,z)\wedge
c(y,z))\,\nu(dz).\end{split}\end{equation} In particular, for any
$x,y,z,u\in\R^d$,
$$c(x,y,u,z)=c(x,y,-u,z-u).$$ The function $c(x,y,u,z)$  and the kernel $\mu_{x,y,u}(dz)$ are  crucial in the construction of the coupling operator below.

For any $x$, $y\in\R^d$ and $\kappa>0$, let
$$
  (x-y)_{\kappa}=\bigg(1\wedge \frac{\kappa}{|x-y|}\bigg)(x-y).
$$
We consider the jump system as follows
$$
  (x,y)\longrightarrow
    \begin{cases}
    (x+z, y+z+(x-y)_\kappa), & \frac12 \mu_{x,y,(y-x)_\kappa}(dz);\\
    (x+z, y+z+(y-x)_\kappa), & \frac12 \mu_{x,y,(x-y)_\kappa}(dz);\\
    (x+z, y+z), & \big(\tilde\nu_{x,y}\! -\! \frac12 \mu_{x,y,(y-x)_\kappa}  \!-\!\frac12 \mu_{x,y,(x-y)_\kappa}\big)(dz);\\
    (x+z,y),& \tilde c(x,y,z)\,\nu(dz);\\
    (x,y+z),& \tilde c(y,x,z)\,\nu(dz),
    \end{cases}
 $$ where
  $$\tilde c(x,y,z) =c(x,z)-c(x,z)\wedge c(y,z).$$
  Furthermore,  for any $h\in C_b^2(\R^{2d})$ and $x,y\in\R^d$, we define the following operator associated with the jumping system above
  \begin{equation}\label{proofth24}\begin{split}
    &\widetilde{L} h(x,y)\cr
    &= \frac{1}{2}\int\Big( h(x+z,y+ z+(x-y)_{\kappa})-h(x,y)-\langle\nabla_xh(x,y), z\rangle \I_{\{|z|\le 1\}}\cr
    &\qquad\qquad -\langle\nabla_yh(x,y), z+(x-y)_{\kappa}\rangle \I_{\{|z+(x-y)_{\kappa}|\le 1\}}\Big)\,\mu_{x,y,(y-x)_{\kappa}}(dz)\cr
    &\quad +\frac{1}{2}\int\Big( h(x+z,y+ z+(y-x)_{\kappa})-h(x,y)-\langle\nabla_xh(x,y), z\rangle \I_{\{|z|\le 1\}}\\
    &\qquad\qquad -\langle\nabla_yh(x,y), z+(y-x)_{\kappa}\rangle \I_{\{|z+(y-x)_{\kappa}|\le 1\}}\Big)\,\mu_{x,y,(x-y)_{\kappa}}(dz)\cr
    &\quad +\int\Big( h(x+z,y+z)-h(x,y)-\langle\nabla_xh(x,y), z\rangle \I_{\{|z|\le 1\}}\cr
    &\qquad\qquad-\langle\nabla_yh(x,y), z\rangle \I_{\{|z|\le 1\}}\Big)\,\Big(\tilde\nu_{x,y} -\frac{1}{2}\mu_{x,y,(x-y)_{\kappa}} -\frac{1}{2}\mu_{x,y,(y-x)_{\kappa}}\Big)(dz)\\
    &\quad +\int\Big( h(x+z,y)-h(x,y)-\langle\nabla_xh(x,y), z\rangle \I_{\{|z|\le 1\}}\Big)\tilde c(x,y,z)\,\nu(dz)\\
    &\quad +\int\Big( h(x,y+z)-h(x,y)-\langle\nabla_yh(x,y), z\rangle \I_{\{|z|\le 1\}}\Big)\tilde c(y,x,z)\,\nu(dz),
  \end{split}\end{equation} where  $\nabla_xh(x,y)$ and $\nabla_yh(x,y)$ are defined as the gradient of $h(x,y)$ with respect to $x\in\R^d$ and $y\in \R^d$ respectively.

We will claim that
\begin{proposition}\label{P:coup1} The operator $\widetilde{L}$ defined by \eqref{proofth24} is indeed a coupling operator of $L$; that is,
 for any $f,g\in C_b^2(\R^d)$, letting $h(x,y)=f(x)+g(y)$ for all $x,y\in\R^d$, it holds that
$$
  \widetilde L h(x,y)= L f(x)+ L g(y).
$$\end{proposition}

\begin{proof} The proof is similar to that in \cite[Section 2.1]{Lwcw} with slight modifications, and for the sake of completeness we present it here.
Let $h(x,y)=g(y)$ for any $x,y\in\R^d$, where $g\in C_b^2(\R^d)$. Then, according to \eqref{proofth24},
  \begin{align*}
    \widetilde{L} h(x,y)
    &= \frac{1}{2}\int \Big(g(y+ z+(x-y)_{\kappa})-g(y)\\
    &\qquad\qquad -\langle\nabla g(y), z+(x-y)_{\kappa}\rangle \I_{\{|z+(x-y)_{\kappa}|\le 1\}}\Big)\,\mu_{x,y,(y-x)_{\kappa}}(dz)\\
    &\quad+\frac{1}{2}\int \Big( g(y+ z+(y-x)_{\kappa})-g(y)\\
    &\qquad\qquad -\langle\nabla g(y), z+(y-x)_{\kappa}\rangle \I_{\{|z+(y-x)_{\kappa}|\le 1\}}\Big)\,\mu_{x,y,(x-y)_{\kappa}}(dz)\\
    &\quad+\int \Big(g(y+z)-g(y)-\langle\nabla g(y), z\rangle \I_{\{|z|\le 1\}}\Big)\\
     &\qquad\qquad\quad\times \Big(\tilde\nu_{x,y} \!-\!\frac{1}{2}\mu_{x,y,(x-y)_{\kappa}} \! -\!\frac{1}{2}\mu_{x,y,(y-x)_{\kappa}}\Big)(dz)\\
    &\quad+\int \Big(g(y+z)-g(y)-\langle\nabla g(y), z\rangle \I_{\{|z|\le 1\}}\Big) \tilde c(y,x,z)\,\nu(dz).
    \end{align*}
Changing the variables $z+(x-y)_{\kappa}\to u$ and $z+(y-x)_{\kappa} \to u$ respectively and using Lemma \ref{L:mmea} below in the first two terms of the right hand side of the equality above lead to
  \begin{align*}
    &\widetilde{L} h(x,y)\\
    &=\frac{1}{2}\int\Big(g(y+u)-g(y)-\langle\nabla g(y), u\rangle \I_{\{|u|\le 1\}}\Big)\\
    &\qquad\qquad \times c(x,y,(y-x)_{\kappa},u-(x-y)_\kappa)\,\nu_{(y-x)_\kappa}(d(u-(x-y)_{\kappa}))\\
    &\quad+\frac{1}{2}\int\Big( g(y+u)-g(y)-\langle\nabla g(y), u\rangle \I_{\{|u|\le 1\}}\Big)\\
    &\qquad\qquad\times  c(x,y,(x-y)_{\kappa},u-(y-x)_\kappa)\,\nu_{(x-y)_\kappa}(d(u-(y-x)_{\kappa}))\\
    &\quad+\int\Big(g(y+z)-g(y)-\langle\nabla g(y), z\rangle \I_{\{|z|\le 1\}}\Big) \\
    &\qquad\qquad\times\Big(\tilde\nu_{x,y} -\frac{1}{2}\mu_{x,y,(x-y)_{\kappa}} -\frac{1}{2}\mu_{x,y,(y-x)_{\kappa}}\Big)(dz)\\
    &\quad+\int \Big(g(y+z)-g(y)-\langle\nabla g(y), z\rangle \I_{\{|z|\le 1\}}\Big) \tilde c(y,x,z)\,\nu(dz)\\
    &=\frac{1}{2}\int\Big(g(y+u)-g(y)-\langle\nabla g(y), u\rangle \I_{\{|u|\le 1\}}\Big) c(x,y,(x-y)_{\kappa},u)\,\nu_{(x-y)_\kappa}(du)\\
    &\quad+\frac{1}{2}\int\Big( g(y+u)-g(y)-\langle\nabla g(y), u\rangle \I_{\{|u|\le 1\}}\Big)c(x,y,(y-x)_{\kappa},u)\,\nu_{(y-x)_\kappa}(du)\\
    &\quad+\int\Big(g(y+z)-g(y)-\langle\nabla g(y), z\rangle \I_{\{|z|\le 1\}}\Big) \\
    &\qquad\qquad\times \Big(\tilde\nu_{x,y} -\frac{1}{2}\mu_{x,y,(x-y)_{\kappa}} -\frac{1}{2}\mu_{x,y,(y-x)_{\kappa}}\Big)(dz)\\
    &\quad+\int \Big(g(y+z)-g(y)-\langle\nabla g(y), z\rangle \I_{\{|z|\le 1\}}\Big) \tilde c(y,x,z)\,\nu(dz)\\
    &=\frac{1}{2}\int\Big(g(y+u)-g(y)-\langle\nabla g(y), u\rangle \I_{\{|u|\le 1\}}\Big)\,\mu_{x,y,(x-y)_{\kappa}}(du)\\
    &\quad+\frac{1}{2}\int\Big( g(y+u)-g(y)-\langle\nabla g(y), u\rangle \I_{\{|u|\le 1\}}\Big)\,\mu_{x,y,(y-x)_{\kappa}}(du)\\
    &\quad+\int\Big(g(y+z)-g(y)-\langle\nabla g(y), z\rangle \I_{\{|z|\le 1\}}\Big)\\
    &\qquad\qquad \times \Big(\tilde\nu_{x,y} -\frac{1}{2}\mu_{x,y,(x-y)_{\kappa}} -\frac{1}{2}\mu_{x,y,(y-x)_{\kappa}}\Big)(dz)\\
    &\quad+\int \Big(g(y+z)-g(y)-\langle\nabla g(y), z\rangle \I_{\{|z|\le 1\}}\Big) \tilde c(y,x,z)\,\nu(dz)\\
    &=L g(y).\end{align*}

On the other hand, if $h(x,y)=f(x)$ for any $x,y\in\R^d$ and any $f\in C_b^2(\R^d)$, then we can easily see that
$\widetilde{L} h(x,y)=Lf(x)$. Combining with both conclusions above yields that the operator $\widetilde L$ defined by \eqref{proofth24} is a coupling operator of $L$. \end{proof}

The following lemma has been used in the proof above.
\begin{lemma}\label{L:mmea} For any $z\in \R^d$ with $z\neq 0$, $\nu_{z}(du)$ is a finite measure on $(\R^d, \mathscr{B}(\R^d))$ such that
$$\mu_{z}(d(u+z))= \mu_{-z}(du)$$
 In particular, $$\mu_{z}(\R^d)= \mu_{-z}(\R^d).$$ \end{lemma}

\begin{proof} The proof has been given in \cite[Remark 2.1 and Corollary 6.2]{Lwcw}. We omit it here.  \end{proof}

\subsection{Coupling process for L\'evy type operators}\label{section2}
The purpose of this part is to construct a coupling process
associated with the coupling operator $ \widetilde{L}$ given by
\eqref{proofth24}. Though the following argument is standard (see
\cite[Section 2]{LW14} or \cite[Section 2.2]{Wang15} for example),
we still would like to present some details here.

 Let $\mathscr{D}([0,\infty);\R^d)$ be the space of right continuous
$\R^d$-valued functions having left limits on $[0,\infty)$ and
equipped with the Skorokhod topology. For $t\ge0$, denote by $X_t$
the projection coordinate map on $\mathscr{D}([0,\infty);\R^d)$. A
probability measure $\Pp^x$ on the Skorokhod space
$\mathscr{D}([0,\infty);\R^d)$ is said to be a solution to the
martingale problem for $(L,C_b^2(\R^d))$ with initial value
$x\in\R^d$, if $\Pp^x(X_0=x)=1$ and, for every $f\in C_b^2(\R^d)$,
$$\left\{f(X_t)-f(x)-\int_0^t Lf(X_s)\,ds, t\ge0\right\}$$
is a $\Pp^x$-martingale. The martingale problem for
$(L,C_b^2(\R^d))$ is said to be well-posed if it has a unique
solution for every initial value $x\in\R^d$. The definitions above are well adapted to the martingale problem for $(\widetilde L, C_b^2(\R^{2d}))$ with necessary modifications.
We can refer to \cite{AK, Bass1, BT, CZ, K, KU, LM, Jp,Jp2, MP92, MP14a, MP14b, St} and the references therein for more details about martingale problem for non-local operators.

In order to prove the existence of the martingale problem for the coupling operator  $(\widetilde L, C_b^2(\R^{2d}))$, we will write
  the coupling operator $\widetilde L$ into the form as the expression of L\'evy type operator on $C_b^2(\R^{2d})$.
For any $x,y\in\R^d$, and $A\in \mathscr{B}(\R^{2d})$, set
  \begin{align*}
\widetilde \nu(x,y,A):=&\frac{1}{2}\int_{\{(z,z-(x-y)_\kappa\in A\}}\,\mu_{x,y,(y-x)_\kappa}(dz)+\frac{1}{2}\int_{\{(z,z+(x-y)_\kappa)\in A\}}\,\mu_{x,y,(x-y)_\kappa}(dz)\\
 &+\int_{\{(z,z)\in A\}}\Big(\tilde \nu_{x,y} -\frac{1}{2}\mu_{x,y,(y-x)_\kappa} -\frac{1}{2}\mu_{x,y,(x-y)_\kappa}\Big)(dz)\\
 &+\int_{\{(z,0)\in A\}} \tilde c(x,y,z)\,\nu(dz)+\int_{\{(0,z)\in A\}}\tilde c(y,x,z)\,\nu(dz) .
   \end{align*}
Then, for any $x,y\in\R^d$ and $f\in C^2_b(\R^{2d})$,
   \begin{align*} \widetilde{L} f(x,y)
   &=\int_{\R^{d}\times \R^d}\Big(f\big((x,y)+(u_1,u_2)\big)-f(x,y)-\langle \nabla_x f(x,y),u_1\rangle\I_{\{|u_1|\le 1\}}\\
  &\qquad\qquad \quad- \langle \nabla_y f(x,y),u_2\rangle\I_{\{|u_2|\le 1\}}\Big)\,\widetilde\nu(x,y,du_1,du_2). \end{align*}

For any $h\in C_b(\R^{2d})$ and $x$, $y\in\R^d$, we have
  \begin{align*}
  &\int_{\R^{2d}} h(u)\frac{|u|^2}{1+|u|^2}\,\widetilde \nu(x,y,du)\\
  &= \int_{\R^{d}\times \R^d} h((u_1, u_2))\frac{|u_1|^2+|u_2|^2}{1+|u_1|^2+|u_2|^2}\,\widetilde \nu(x,y,du_1,du_2)\\
  &=\frac{1}{2}\int_{\R^{d}}
h((z,z+(x-y)_\kappa))
\frac{|z|^2+|z+(x-y)_\kappa|^2}{1+|z|^2+|z+(x-y)_\kappa|^2}\,\mu_{x,y,(y-x)_\kappa}(dz)\\
&\quad+ \frac{1}{2}\int_{\R^{d}} h((z,z-(x-y)_\kappa))
\frac{|z|^2+|z-(x-y)_\kappa|^2}{1+|z|^2+|z-(x-y)_\kappa|^2}\,\mu_{x,y,(x-y)_\kappa}(dz)\\
&\quad+ \int_{\R^{d}} h((z,z))
\frac{|z|^2+|z|^2}{1+|z|^2+|z|^2}\Big(\tilde\nu_{x,y}
-\frac{1}{2}\mu_{x,y,(x-y)_\kappa}
-\frac{1}{2}\mu_{x,y,(y-x)_\kappa}\Big)(dz)\\
&\quad+\int_{\R^d}
h((z,0))\frac{|z|^2}{1+|z|^2}\big(c(x,z)-c(x,z)\wedge
c(y,z)\big)\,\nu(dz)\\
&\quad+\int_{\R^d}
h((0,z))\frac{|z|^2}{1+|z|^2}\big(c(y,z)-c(x,z)\wedge
c(y,z)\big)\,\nu(dz).
\end{align*} Since $c(x,z)$ is bounded and $(x,z)\mapsto c(x,z)$ is continuous, the
function $(x,y)\mapsto\int_{\R^{2d}} h(u)\frac{|u|^2}{1+|u|^2}\,\widetilde
\nu(x,y,du)$ is continuous too.  Therefore, by \cite[Theorem
2.2]{St}, there is a solution to the martingale problem for
$(\widetilde{L}, C_b^2(\R^{2d}))$, i.e., there are a probability
space $(\widetilde{\Omega}, \widetilde{\mathscr{F}},
(\widetilde{\mathscr{F}}_t)_{t\ge0}, \widetilde{\Pp})$ and an
$\overline{\R}^{2d}:=(\R\cup\{\infty\})^{2d}$-valued process
$(\widetilde{X}_t)_{t\ge0}:=(X_t',X_t'')_{t\ge0}$ such that
$(\widetilde{X}_t)_{t\ge0}$ is
$(\widetilde{\mathscr{F}}_t)_{t\ge0}$-progressively measurable, and
for every $f\in C_b^2(\R^{2d})$,
  $$\bigg\{f(\widetilde{X}_{t\wedge \widetilde\zeta})-\int_0^{t\wedge \widetilde\zeta} \widetilde{L}f(\widetilde{X}_s)\,ds, t\ge 0\bigg\}$$
is an $(\widetilde{\mathscr{F}}_t)_{t\ge0}$-local martingale, where
$\widetilde\zeta$ is the explosion time of $(\widetilde{X}_t)_{t\ge0}$, i.e.,\
  $$\widetilde\zeta=\lim_{n\to\infty}\inf\Big\{t\ge0: |X_t'|+|X_t''|\ge n\Big\}.$$
By Proposition \ref{P:coup1}, $\widetilde{L}$ is the coupling
operator of $L$, and so both distributions of the processes
$(X_t')_{t\ge0}$ and $(X_t'')_{t\ge0}$ are solutions to the
martingale problem of $L$. Since we assume that the martingale
problem for $(L,C_b^2(\R^d))$ is well-posed, the processes
$(X_t')_{t\ge0}$ and $(X_t'')_{t\ge0}$ are non-explosive, and so we
have $\widetilde\zeta=\infty$ a.e. That is, the coupling operator
$\widetilde{L}$ generates a non-explosive process
$(\widetilde{X}_t)_{t\ge0}$.

Let $T$ be the coupling time of $(X_t')_{t\ge0}$ and
$(X_t'')_{t\ge0}$, i.e.,\
  $T=\inf\{t\ge0: X_t'=X_t''\}.$
Then $T$ is an $(\widetilde{\mathscr{F}}_t)_{t\geq 0}$-stopping
time. Construct a new process $(Y'_t)_{t\ge0}$ as follows
  $$Y_t'=
  \begin{cases}
  X_t'', & t< T;\\
  X_t', & t\ge T.\\
  \end{cases}
  $$
We can verify that $(Y_t')_{t\ge0}$ is a solution to the martingale problem of $L$, see \cite[Section 2.2]{Wang15}. Since
the martingale problem for the operator $L$ is well posed,
$(Y_t')_{t\ge0}$ and $(X_t'')_{t\ge0}$ are equal in the
distribution. Therefore, we conclude that $(X_t',Y_t')_{t\ge0}$ is
also a non-explosive coupling process of $(X_t)_{t\ge0}$ such that
$X_t'=Y_t'$ for any $t\ge T$ and the generator of
$(X_t',Y_t')_{t\ge0}$ before the coupling time $T$ is just the
coupling operator $\widetilde{L}$.
\section{Coupling approach for regularity of semigroups}\label{section3}

\subsection{Preliminary calculations}\label{sec-pc} In the following, we assume that $\kappa\in(0,1]$.
Let $\widetilde{L}$ be the coupling operator given above. We will estimate $\widetilde{L} f(|x-y|)$
for any $0\le f\in C_b([0,\infty))\cap C^2((0,\infty))$ such that $f(0)=0$, and $f'\ge0$ and $f''\le0$ on $(0,2]$.

For any $h\in C_b^2(\R^{2d})$, define \begin{align*} \widetilde L_R h(x,y)= & \int \Big(h(x+z,y)-h(x,y)-\langle\nabla_x h(x,y), z\rangle\I_{\{|z|\le 1\}}\Big)\tilde c(x,y,z)\,\nu(dz)\\
&+ \int \Big(h(x,y+z)-h(x,y)-\langle \nabla_y h(x,y), z\rangle\I_{\{|z|\le 1\}}\Big)\tilde c(y,x,z)\,\nu(dz).\end{align*} Set $$\widetilde L_C:= \widetilde L-
\widetilde L_R.$$ Note that, if we define the
following operator
\begin{align*}L_Cf(x)&:=\int \big(f(x+z)-f(x)-\langle \nabla f(x),z\rangle\I_{B(0,1)}(z)\big)\,\nu_{x,y}(dz)\\
&=\int \big (f(x+z)-f(x)-\langle \nabla f(x),z\rangle\I_{B(0,1)}(z)\big)(c(x,z)\wedge c(y,z))\,\nu(dz),\end{align*} then, following the proof of Proposition \ref{P:coup1}, we can see that $\widetilde L_C$ is a coupling operator of $L_C$.

First, according to Lemma \ref{L:mmea}, $\nu_{(y-x)_\kappa}(dz)$ is
a finite measure on $\R^d$ for any $x,y\in \R^d$ with $x\neq y$. Then, for any $f\in C_b([0,\infty))\cap
C^2((0,\infty))$ and $x,y\in \R^d$ with $x\neq y$,
\begin{align*}&\int\big(\langle\nabla_xf(|x-y|),z\rangle \I_{\{|z|\le 1\}}\\
&\qquad\qquad\qquad +\langle\nabla_y
f(|x-y|),z+(x-y)_\kappa\rangle\I_{\{|
z+(x-y)_\kappa|\le 1\}}\big)\,\mu_{x,y,(y-x)_\kappa}(dz)\\
&=\frac{f'(|x-y|)}{|x-y|}\bigg(\int_{\{|z|\le 1\}}\langle
x-y,z\rangle\,\mu_{x,y,(y-x)_\kappa}(dz)\\
&\qquad\qquad\qquad\quad- \int_{\{|z+(x-y)_\kappa|\le 1\}}\langle
x-y,z+(x-y)_\kappa\rangle\,\mu_{x,y,(y-x)_\kappa}(dz)\bigg)\\
&= \frac{f'(|x-y|)}{|x-y|}\bigg(\int_{\{|z|\le 1\}}\langle
x-y,z\rangle\,\mu_{x,y,(y-x)_\kappa}(dz)\\
&\qquad\qquad\qquad\quad- \int_{\{|z|\le 1\}}\langle
x-y,z\rangle\,\mu_{x,y,(x-y)_\kappa}(dz)\bigg),\end{align*} where in the last equality we used the fact that
$$c(x,y,(y-x)_\kappa, u-(x-y)_\kappa)=c(x,y,(x-y)_\kappa,u),\quad x,y,u\in \R^d$$ and Lemma \ref{L:mmea}.
 Similarly,
it holds that
\begin{align*}&\int\big(\langle\nabla_xf(|x-y|),z\rangle \I_{\{|z|\le 1\}}\\
&\qquad\qquad\qquad +\langle\nabla_y
f(|x-y|),z+(x-y)_\kappa\rangle\I_{\{|
z+(x-y)_\kappa|\le 1\}}\big)\,\mu_{x,y,(x-y)_\kappa}(dz)\\
&= \frac{f'(|x-y|)}{|x-y|}\bigg(\int_{\{|z|\le 1\}}\langle
x-y,z\rangle\,\mu_{x,y,(x-y)_\kappa}(dz)\\
&\qquad\qquad\qquad\quad- \int_{\{|z|\le 1\}}\langle
x-y,z\rangle\,\mu_{x,y,(y-x)_\kappa}(dz)\bigg).\end{align*} Therefore,
for any $x$, $y\in\R^d$ with $x\neq y$,
\begin{align*}
\widetilde{L}_C f(|x-y|)   &=\frac{1}{2}\mu_{x,y,(x-y)_\kappa}(\R^d) \left(f\big(|x-y|+\kappa\wedge |x-y|\big)-f(|x-y|)\right)\\
   &\quad+ \frac{1}{2}\mu_{x,y,(y-x)_\kappa}(\R^d)\left(f\big(|x-y|-\kappa\wedge |x-y|\big)-f(|x-y|)\right)\\
  &=\frac{1}{2} \mu_{x,y,(x-y)_{\kappa}}(\R^d)\Big[f\big(|x-y|+\kappa\wedge |x-y|\big)\\
  &\qquad\qquad\qquad\qquad\,\,+f\big(|x-y|-\kappa\wedge |x-y|\big)  - 2f(|x-y|)\Big],
  \end{align*} where in the last equality we have used the fact that $\mu_{x,y,(x-y)_\kappa}(\R^d)=\mu_{x,y,(y-x)_\kappa}(\R^d).$

Next, we assume that $f\ge 0$ with $f(0)=0$ on $[0,\infty)$, and
$f'\ge0$ and $f''\le 0$ on $(0,2]$. Let $\varepsilon_0\in
(0,\kappa]$. Then, for any $\varepsilon\in (0,\varepsilon_0]$ and
any $x,y\in \R^d$ with $|x-y|\le \varepsilon$, we have
 \begin{equation}\label{proofth2544} \widetilde{L}_C f(|x-y|)\le \frac{1}{2}J(|x-y|) \Big(f(2|x-y|)- 2f(|x-y|)\Big),  \end{equation} where
in the inequality above \begin{equation}\label{refmea}J(r):=\inf_{x,y\in \R^d: |x-y|=r}
\mu_{x,y,(x-y)}(\R^d)\end{equation} and we have used the fact that $$f(2r)=f(r)+\int_r^{2r} f'(s)\,ds=f(r)+\int_0^r f'(s+r)\,ds\le f(r)+\int_0^r f'(s)\,ds=
2f(r)$$ for any $r\in (0,\varepsilon_0]$.

We will give estimates for $\widetilde{L}_R f(|x-y|)$. For any $f\in
C_b([0,\infty))\cap C^2((0,\infty))$ with $f\ge0$, $f'\ge0$ and
$f''\le 0$ on $(0,2]$. Then, for any $x$, $y\in\R^d$ with $x\neq y$,
\begin{align*}\allowdisplaybreaks
    \widetilde{L}_R &f(|x-y|)\\
    =& \int\bigg(f(|x-y+z|)- f(|x-y|)-\frac{f'(|x-y|)}{|x-y|}\langle x-y, z\rangle \I_{\{|z|\le 1\}}\bigg)\\
     &\qquad\qquad\times (c(x,z)-c(x,z)\wedge c(y,z))\,\nu(dz)\\
     &+ \int\bigg(f(|x-y-z|)- f(|x-y|)+\frac{f'(|x-y|)}{|x-y|}\langle x-y, z\rangle \I_{\{|z|\le 1\}}\bigg) \\
     &\qquad\qquad \quad \times (c(y,z)-c(x,z)\wedge c(y,z))\,\nu(dz)\\
     =& \int_{\{|z|\le 1\}}\bigg(f(|x-y+z|)- f(|x-y|)-\frac{f'(|x-y|)}{|x-y|}\langle x-y, z\rangle\bigg)\\
     &\qquad\qquad\times (c(x,z)-c(x,z)\wedge c(y,z))\,\nu(dz)\\
     &+ \int_{\{|z|\le 1\}}\bigg(f(|x-y-z|)- f(|x-y|)+\frac{f'(|x-y|)}{|x-y|}\langle x-y, z\rangle \bigg) \\
     &\qquad\qquad \quad \times (c(y,z)-c(x,z)\wedge c(y,z))\,\nu(dz) \\
     &+\int_{\{|z|>1\}} \bigg(f(|x-y+z|)- f(|x-y|)\bigg)(c(x,z)-c(x,z)\wedge c(y,z))\,\nu(dz)\\
     &+\int_{\{|z|>1\}} \bigg(f(|x-y-z|)- f(|x-y|)\bigg)(c(y,z)-c(x,z)\wedge c(y,z))\,\nu(dz).
  \end{align*}

   We further consider the following two cases.

(i) Since for any $a,b\in(0,2]$,
$$f(b)-f(a)\le f'(a) (b-a),$$ we have that for any $x,y,z\in \R^d$ with $0<|x-y|\le \varepsilon_0$ and $|z|\le 1$,
\begin{align*}&f(|x-y+z|)- f(|x-y|)-\frac{f'(|x-y|)}{|x-y|}\langle x-y, z\rangle\\
&\le \frac{f'(|x-y|)}{|x-y|}\Big(|x-y+z||x-y|- |x-y|^2-\langle x-y, z\rangle\Big)\\
&\le  \frac{f'(|x-y|)}{|x-y|}\frac{|z|^2}{2}, \end{align*} where the last inequality follows from the fact that
$$\langle x-y,z\rangle=\frac{1}{2}\left(|x-y+z|^2-|x-y|^2-|z|^2\right),\quad x,y,z\in\R^d.$$ This yields that for any $x,y\in \R^d$ with $0<|x-y|\le \varepsilon_0$,
\begin{equation}\label{proofth2544}\begin{split}\widetilde L_R f(|x-y|)&\le  \bigg(\int_{\{|z|\le 1\}} |c(x,z)-c(y,z)||z|^2\,\nu(dz) \bigg) \frac{f'(|x-y|)}{|x-y|}\\
&\quad + 2\bigg(\int_{\{|z|>1\}}\,\nu(dz) \bigg)\bigg[\sup_{x,z\in
\R^d: |z|> 1} c(x,z)\bigg]\|f\|_\infty .\end{split}\end{equation}

(ii) If \begin{equation}\label{e:ffvv}\int_{\{|z|\le
1\}}|z|\,\nu(dz)<\infty,\end{equation} then for any $x,y,z\in \R^d$ with $0<|x-y|\le \varepsilon_0$ and $|z|\le 1$,
\begin{align*}&f(|x-y+z|)- f(|x-y|)-\frac{f'(|x-y|)}{|x-y|}\langle x-y, z\rangle\le 2{f'(|x-y|)}|z|. \end{align*} Following the same argument as that in (i), we can arrive at for any $x,y\in \R^d$ with $0<|x-y|\le \varepsilon_0$
\begin{equation}\label{proofth25441}\begin{split}\widetilde L_R f(|x-y|)\le &4\bigg(\int_{\{|z|\le 1\}} |c(x,z)-c(y,z)| |z|\,\nu(dz) \bigg)
{f'(|x-y|)}\\
&+2\bigg( \int_{\{|z|>1\}}\,\nu(dz)\bigg) \bigg[\sup_{x,z\in \R^d:
|z|> 1} c(x,z)\bigg]\|f\|_\infty .\end{split}\end{equation}

Combining all the estimates above, we can get that
\begin{proposition}\label{P:eses}
Let $0\le f\in C_b([0,\infty))\cap C^2((0,\infty))$ such that
$f(0)=0$, and $f'\ge0$ and $f''\le 0$ on $(0,2]$. Let
$0<\varepsilon_0\le \kappa\le 1$ and $J(r)$ be defined by
\eqref{refmea}. Then, for any $x,y\in \R^d$ with $0<|x-y|\le
\varepsilon_0$,
\begin{itemize}
\item[(1)] it holds that
\begin{align*} \widetilde{L} f(|x-y|)\le  &\frac{1}{2}J(|x-y|) \Big(f(2|x-y|)- 2f(|x-y|)\Big)\\
&+ \bigg(\int_{\{|z|\le 1\}} |c(x,z)-c(y,z)||z|^2\,\nu(dz) \bigg)\frac{f'(|x-y|)}{|x-y|}\\
&+ 2\bigg(\int_{\{|z|>1\}}\,\nu(dz) \bigg)\bigg[\sup_{x,z\in \R^d:
|z|> 1} c(x,z)\bigg]\|f\|_\infty.\end{align*}
\item[(2)] if additionally \eqref{e:ffvv} is satisfied, then
\begin{align*} \widetilde{L} f(|x-y|)\le  &\frac{1}{2}J(|x-y|) \Big(f(2|x-y|)- 2f(|x-y|)\Big)\\
&+ 4 \bigg(\int_{\{|z|\le 1\}} |c(x,z)-c(y,z)||z|\,\nu(dz) \bigg){f'(|x-y|)}\\
&+ 2\bigg(\int_{\{|z|>1\}}\,\nu(dz) \bigg)\bigg[\sup_{x,z\in
\R^d:|z|> 1} c(x,z)\bigg]\|f\|_\infty.\end{align*}\end{itemize}
\end{proposition}

\begin{remark} The estimates for $\widetilde{L} f(|x-y|)$ consist three terms. The first one comes from the operator $\widetilde{L}_C$,
which is a leading part for our purpose. Other two terms are due to
the operator $\widetilde{L}_R$.\end{remark}
\subsection{General results}
In the following, we present general results concerning the spatial regularity of semigroups.
\begin{theorem}\label{T:re} Assume that there is  a
nonnegative and $C_b([0,\infty))\cap C^3((0,\infty))$-function $\psi$ such that
\begin{itemize}
\item[(i)] $\psi(0)=0$, $\psi'\ge0$, $\psi''\le0$ and $\psi'''\ge0$ on $(0,2]$;
\item[(ii)] For any constants $c_1,c_2>0$,
\begin{equation}\label{e:fe1}\limsup_{r\to0}\bigg[J_\nu(r)r^2\psi''(2r)+c_1w(r)\psi'(r)r^{-1}+c_2\bigg]<0,\end{equation} where
\begin{equation}\label{e:ffgg1}J_\nu(r)= \inf_{z\in \R^d: |z|=r} \nu_{z}(\R^d)\end{equation} with $\nu_z(dz)$ defined in \eqref{e:not},  and
$$w(r)=\sup_{x,y\in \R^d: |x-y|=r} \bigg(\int_{\{|z|\le 1\}}
|c(x,z)-c(y,z)||z|^2\,\nu(dz) \bigg).$$
\end{itemize}Then, there are constants $C,\varepsilon_0>0$ such that for all $f\in B_b(\R^d)$ and $t>0$,
\begin{equation}\label{e:fefe1}\sup_{x\neq y}\frac{|P_tf(x)-P_tf(y)|}{\psi(|x-y|)}\le C\|f\|_\infty
\inf_{\varepsilon\in(0,\varepsilon_0]}\bigg[\frac{1}{\psi(\varepsilon)}+\frac{1}{t\lambda_\psi(\varepsilon)}\bigg],\end{equation}
where
$$\lambda_\psi(\varepsilon):=-\sup_{0<r\le \varepsilon}J_\nu(r)r^2\psi''(2r).$$

Assume additionally that \eqref{e:ffvv} holds. Then the conclusion above still holds, if \eqref{e:fe1} is replaced by
\begin{equation*}\label{e:fe11}\limsup_{r\to0}\bigg[J_\nu(r)r^2\psi''(2r)+c_1w_*(r)\psi'(r)+c_2\bigg]<0,\end{equation*} where
$$w_*(r)=\sup_{x,y\in \R^d: |x-y|=r} \bigg(\int_{\{|z|\le 1\}}
|c(x,z)-c(y,z)||z|\,\nu(dz) \bigg).$$  \end{theorem}

\begin{proof} First, by \eqref{e:fe1}, we have
$$\limsup_{r\to 0}J_\nu(r)r^2\psi''(2r)<0.$$
Due to $\psi'''\geq 0$ on $(0,2]$, it holds for any $0<r\le 1$ that
  $$2\psi(r)-\psi(2r)=-\int_0^r\int_s^{r+s} \psi''(u)\,du\,ds\ge -\psi''(2r)
  r^2.$$
  Let $0<\varepsilon_0<\kappa\le 1$ and
  $\varepsilon\in(0,\varepsilon_0]$. For any $x,y\in \R^d$ with
  $0<|x-y|\le \varepsilon$, according to Proposition
  \ref{P:eses}(1), we find that
  \begin{align*}\widetilde{L}\psi(|x-y|)&\le  \frac{1}{2}J(|x-y|)\psi''(2|x-y|)|x-y|^2+ c_1w(|x-y|)\frac{\psi'(|x-y|)}{|x-y|}+c_2,\end{align*}
  where in the inequality above we used the facts that $c(x,z)$ is
  bounded from above and $\psi$ is bounded.

  Since $c(x,z)$ is bounded from below, there is a constant $c_3>0$
  such that for all $r>0$,
  $J(r)\ge c_3J_\nu(r).$ This further yields that for any $x,y\in \R^d$ with
  $0<|x-y|\le \varepsilon$
  \begin{align*}\widetilde{L}\psi(|x-y|)&\le  \frac{c_3}{2}J_\nu(|x-y|)\psi''(2|x-y|)|x-y|^2+ c_1w(|x-y|)\frac{\psi'(|x-y|)}{|x-y|}+c_2\\
  &\le c_4J_\nu(|x-y|)\psi''(2|x-y|)|x-y|^2\\
  &\le -c_4 \lambda_\psi(\varepsilon) ,\end{align*} where in the
  second inequality we used \eqref{e:fe1}.

  Having the inequality above at hand, we can obtain the first desired assertion by \cite[Proposition 4.1]{MW}.
The second desired assertion follows from the argument above and Proposition \ref{P:eses}(2). \end{proof}

From Theorem \ref{T:re}, we can further deduce the time-space
regularity of semigroups. In details, let $(X_t^x)_{t\ge0}$ be the
strong Markov process associated with the operator $L$ starting from
$x$. Under assumptions of Theorem \ref{T:re}, for any $0<s<t$,
$x,y\in \R^d$ and  $f\in B_b(\R^d)$,
\begin{align*}|P_sf(x)-P_tf(y)|=&|\Ee^xf(X_s)-\Ee^yf(X_t)|=|\Ee^xf(X_s)-\Ee^y \Ee^{X_{t-s}^y} f(X_s)|\\
\le&\Ee^y|\Ee^xf(X_s)-\Ee^{X_{t-s}^y} f(X_s)|\\
\le & C \|f\|_\infty
\inf_{\varepsilon\in(0,\varepsilon_0]}\bigg[\frac{1}{\psi(\varepsilon)}+\frac{1}{t\lambda_\psi(\varepsilon)}\bigg]
\Ee^y\psi(|x-X_{t-s}^y|),\end{align*} where in the second equality
we used the Markov property, and the last inequality follows from
\eqref{e:fefe1}. In order to estimate $\Ee^y\psi(|x-X_{t-s}^y|)$,
one can refer to \cite{Kul} for the recent study of moments
estimates for L\'evy-type processes. The details are omitted here.

\subsection{Proofs}
To prove Theorem \ref{main}, we also need the following lemma.

\begin{lemma} Suppose that there are constants $\alpha\in (0,2)$ and $c_0>0$ such that
$$\nu(dz)\ge \frac{c_0}{|z|^{d+\alpha}}\I_{V_\xi}(dz),$$ where $V_\xi$ is defined by \eqref{e:mea}. Then, there are constants $c_1>0$ and $r_0\in(0,1)$ such that for all $0<r\le r_0$,
$$J_\nu(r)\ge c_1r^{-\alpha},$$ where $J_\nu$ is defined by \eqref{e:ffgg1}.
\end{lemma}

\begin{proof} For any $z\in \R^d$, let $z=(z_1, z_2,\cdots,z_d)$.
Without loss of generality, we may and can assume that $\xi=e_1=(1,0,\cdots,0)$.
Denote by $$q(z)=\frac{c_0}{|z|^{d+\alpha}}\I_{\{z_1\ge \delta |z|, |z|\le 1\}}(z),\quad z\in \R^d.$$ Then, for any $x,z\in \R^d$,
\begin{align*} q(z)\wedge q(x+z)=&\I_{\{z_1\ge \delta |z|, z_1+x_1\ge \delta |x+z|, |z|\le 1, |x+z|\le 1\}}\left( \frac{c_0}{|z|^{d+\alpha}}\wedge \frac{c_0}{|z+x|^{d+\alpha}}\right).\end{align*}

In the following, we first suppose that $x_1\ge0$. Hence, for any $x\in \R^d$ with $|x|$ small enough,
\begin{align*}&\int q(z)\wedge q(x+z)\,dz\\
&\ge \int_{\{z_1\ge (1+\delta)|z|/2 , 2\delta|x|/(1-\delta)\le |z|\le \delta/(1+\delta)\}}\frac{c_0}{(|x|+|z|)^{d+\alpha}}\,dz\\
&\ge c_1 \int_{\{z_1\ge (1+\delta)|z|/2 , 2\delta|x|/(1-\delta)\le |z|\le \delta/(1+\delta)\}}\frac{1}{|z|^{d+\alpha}}\,dz\\
&\ge c_1 \int_{\{(1+\delta)(z_2^2+\cdots+z_d^2)^{1/2}/(1-\delta)\le z_1\le \delta/[2(1+\delta)], 2\delta|x|/(1-\delta)\le (z_2^2+\cdots+z_d^2)^{1/2}\le \delta/[2(1+\delta)]\}}\frac{1}{|z|^{d+\alpha}}\,dz\\
&\ge c_2\int_{2\delta|x|/(1-\delta)}^{c_3} \,dr \int_{ (1+\delta) r/(1-\delta)}^{c_4} \frac{r^{d-2}}{(z_1+r)^{d+\alpha}}\,dz_1\\
&\ge c_5 |x|^{-\alpha}.
\end{align*}

If $x_1\le 0$, then, following the argument above, we have
\begin{align*}&\int q(z)\wedge q(x+z)\,dz\\
&\ge \int_{\{z_1\ge (1+\delta) |z|/2, 2(1+\delta)|x|/(1-\delta)\le |z|\le (1+\delta)/(3+\delta)\}}\frac{c_0}{(|x|+|z|)^{d+\alpha}}\,dz\\
&\ge c_6 |x|^{-\alpha}.
\end{align*}

Combining all the estimates above, we have obtained the desired assertion.
\end{proof}

We are now in a position to give

\begin{proof}[Proof of Theorem $\ref{main}$]
We will apply Theorem \ref{T:re}. For (1), noticing that
$\alpha_2\in(1,2)$ too, we choose $\phi(r)=r(1-\log^{-\theta}(1/r))$
for $r>0$ small enough, where $\theta$ is given in
\eqref{e:ffeesss}. For (2), we take $\phi(r)=r \log^\theta(1/r)$ for
$r>0$ small enough, where $\theta>0$ if $\alpha_1>1$, and $\theta>1$
if $\theta_1=1$. For (3), since $\alpha_2\in (0,1)$, \eqref{e:ffvv}
holds. Then, we can take $\phi(r)=r^{\theta}$ for $r>0$ small
enough. Therefore, with functions $\phi$ above, the desired
assertion follows from Theorem \ref{T:re}.
\end{proof}

\begin{proof}[Proof of Corollary $\ref{C:cor1}$] Let $u$ be a
bounded harmonic function on $\R^d$. Then, for any $x,y\in \R^d$,
$$u(x)-u(y)=P_1u(x)-P_1u(y).$$ This along with Theorem \ref{main}
immediately yields the desired assertion. \end{proof}

\begin{proof}[Proof of Corollary $\ref{C:cor2}$] Since the proof is mainly based on that of \cite[Theorem 2.1]{RS}, we only point out necessary modifications here. For simplicity,
we just consider the case (3) in Theorem \ref{main}.  For any
$\rho\ge1$, let $v(x)=\rho^{-\beta} u(\rho x)$. Then, it is obvious
that $P_tv(x)=v(x)$ for all $x\in \R^d$ and $t>0$; moreover,
$\|v\|_{L^\infty(B_R(0),dx)}\le c R^\beta$ with the same constant as
$u$.

For any $M>0$, let $v_M(x)=v(x)\I_{\{|x|\le M\}}$. Then, for any
$x,y\in \R^d$,  according to Theorem \ref{main}(3), we have
\begin{align*}|v(x)-v(y)|=&|P_1v(x)-P_1v(y)|\\
\le
&|P_1v_M(x)-P_1v_M(y)|+|P_1(v-v_M)(x)|+|P_1(v-v_M)(y)|\\
\le & c_1|x-y|^\theta
M^\beta\\
&+|P_1v_M(x)-P_1v_M(y)|+|P_1(v-v_M)(x)|+|P_1(v-v_M)(y)|.\end{align*}

On the other hand, under \eqref{e:mee} (in particular, $\nu(dz)\le
\frac{c_2}{|z|^{d+\alpha_2}}\,dz$), for any $\varepsilon>0$ and the
function $h(x)=(1+|x|^2)^{(\alpha_2-\varepsilon)/2}$, we can check
that there is a constant $c_3>0$ such that for all $x\in \R^d$,
$$Lh(x)\le c_3h(x),$$ which yields that $P_1 h(x)\le c_4 h(x).$
Thus, for all $x\in B_1(0)$
\begin{align*}|P_1(v-v_M)(x)|&\le \int_{\{|z|>M\}} |v|(z)\,P_1(x,dz)\\
&\le \left(\int_{\{|z|>M\}} h(z)\,P_1(x,dz)\right) \left(\sup_{z\in
M} \frac{ |v|(z)}{h(z)}\right)\\
&\le c_5 M^{-(\alpha_2-\beta-\varepsilon)}.
\end{align*} In particular, taking $\varepsilon=(\alpha_2-\beta)/2$,
we arrive at that for all $x\in B_1(0)$,
$$|P_1(v-v_M)(x)|\le c_5 M^{-\varepsilon}.$$

Combining with all the estimates above, for any $x,y\in B_1(0)$,
$$|v(x)-v(y)|\le c_6\left(|x-y|^\theta
M^\beta-M^{-\varepsilon}\right).$$ Letting
$M=|x-y|^{-\theta/(\beta+\varepsilon)}$, we get that for any $x,y
\in B_1(0)$,
$$|v(x)-v(y)|\le c_7 |x-y|^\gamma,$$ where $\gamma=\varepsilon
\theta/(\beta+\varepsilon).$ This shows the same conclusion as
\cite[(2.9)]{RS}. Furthermore, one can follow  the argument of
\cite[Theorem 2.1]{RS} to prove the desired assertion.
  \end{proof}

Next, we present the
\begin{proof}[Proof of Theorem $\ref{main0}$] Let $L$ be the operator given by \eqref{levytype}, and let $L_\mu:=L_*-L$.
Since $L_*$ is a linear operator, we can split the construction of
coupling operator $L_*$ into those of $L$ and $L_\mu$. For $L$, we
still use the coupling operator $\widetilde L$ defined by
\eqref{proofth24}. For $L_\mu$, we define a coupling operator $
\widetilde{L}_\mu$ as follows: for any $h\in C_b^2(\R^{2d})$,
 \begin{align*}
    \widetilde{L}_\mu h(x,y) &= \int\Big( h(x+z,y+z)-h(x,y)-\langle\nabla_xh(x,y), z\rangle \I_{\{|z|\le 1\}}\cr
    &\qquad\qquad-\langle\nabla_yh(x,y), z\rangle \I_{\{|z|\le 1\}}\Big)\,\mu(x,dz)\wedge \mu(y,dz)\\
    &\quad +\int\Big( h(x+z,y)-h(x,y)-\langle\nabla_xh(x,y), z\rangle \I_{\{|z|\le 1\}}\Big)\\
    &\qquad\qquad\qquad \times \,\left(\mu(x,dz)- \mu(x,dz)\wedge \mu(y,dz)\right)\\
    &\quad +\int\Big( h(x,y+z)-h(x,y)-\langle\nabla_yh(x,y), z\rangle \I_{\{|z|\le 1\}}\Big)\\
       &\qquad\qquad\qquad \times \,\left(\mu(y,dz)- \mu(x,dz)\wedge \mu(y,dz)\right).
  \end{align*} Then, we can follow the proof of Proposition \ref{P:coup1} and obtain that $\widetilde L_*:=\widetilde L+ \widetilde{L}_\mu$ is a coupling operator of $L$. Furthermore, using the assumption that for any $h\in C_b^2(\R^d)$, the function $x\mapsto \int_{\R^d} h(z) \frac{|z|^2}{1+|z|^2}\,\mu(x,dz)$ is continuous, and repeating the argument in Section \ref{section2}, we can prove the existence of coupling process associated with the coupling operator $\widetilde L_*$ above.

Next, for any $0\le f\in C_b([0,\infty))\cap C^2((0,\infty))$ such that $f(0)=0$, and $f'\ge0$ and $f''\le0$ on $(0,2]$, we will adopt the arguments of \eqref{proofth2544}  and \eqref{proofth25441} to obtain some similar estimates about $\widetilde{L}_\mu f(|x-y|)$. With these at hand and estimates for $\widetilde L$ in Proposition \ref{P:eses}, one can obtain the desired conclusion by using Theorem \ref{T:re} and following the proof of Theorem \ref{main} line by line.
\end{proof}


\medskip

\noindent \textbf{Acknowledgements.}  The research is supported by
National Natural Science Foundation of China (No.\ 11522106), Fok
Ying Tung Education Foundation (No.\ 151002), and  the Program for
Nonlinear Analysis and Its Applications (No. IRTL1206).

\end{document}